\documentclass[a4paper]{amsart}
\usepackage{amsmath,amsthm,amssymb,latexsym,epic,bbm,comment,mathrsfs}
\usepackage{graphicx,enumerate,stmaryrd,color,tikz,ytableau}
\usepackage[all,2cell]{xy}
\xyoption{2cell}
\usetikzlibrary{patterns,snakes}

\newtheorem{theorem}{Theorem}
\newtheorem{lemma}[theorem]{Lemma}
\newtheorem{corollary}[theorem]{Corollary}
\newtheorem{proposition}[theorem]{Proposition}

\newtheorem{remark}[theorem]{Remark}
\usepackage[all]{xy}
\usepackage[active]{srcltx}
\usepackage[parfill]{parskip}
\usepackage{enumerate}

\begin{document}
\title[A stability phenomenon in Kazhdan-Lusztig combinatorics]
{A stability phenomenon in \\ Kazhdan-Lusztig combinatorics}

\author[Samuel Creedon and Volodymyr Mazorchuk]
{Samuel Creedon and Volodymyr Mazorchuk}

\begin{abstract}
We prove that, when $n$ goes to infinity, the expression,
with respect to the dual Kazhdan-Lusztig basis, of
the product $\hat{\underline{H}}_x\underline{H}_y$
of elements of the dual and the usual Kazhdan-Lusztig bases
in the Hecke algebra of the symmetric group
$S_n$ stabilizes. As an application, we define the
action of projective functors on the principal block 
of category $\mathcal{O}$  for $\mathfrak{sl}_\infty$
and show that the subcategory of finite length objects 
is stable under this action. As a bonus, we also prove that this
latter block is Koszul, answering, for this block,  
a question from \cite{CP}.
\end{abstract}

\maketitle

\section{Introduction and description of the results}\label{s1}

Let $\mathfrak{g}$ be a semi-simple complex Lie algebra. 
To a fixed triangular decomposition 
\begin{displaymath}
\mathfrak{g}=\mathfrak{n}_-\oplus
\mathfrak{h}\oplus\mathfrak{n}_+ 
\end{displaymath}
of $\mathfrak{g}$, we can associate 
the corresponding Bernstein-Gelfand-Gelfand category $\mathcal{O}$, see \cite{BGG,Hu}.
The principal block $\mathcal{O}_0$ of $\mathcal{O}$ is equivalent to 
the category of finite dimensional modules over some finite
dimensional associative algebra. This block is equipped with 
an action of the  monoidal category $\mathscr{P}$ of projective
endofunctors, see \cite{BG}. 

Both $\mathcal{O}_0$ and $\mathscr{P}$
admit $\mathbb{Z}$-graded lifts, ${}^{\mathbb{Z}}\mathcal{O}_0$ and 
${}^{\mathbb{Z}}\mathscr{P}$, respectively. Taking the Gro\-then\-di\-eck group, 
resp. ring, of the two latter categories, one obtains the right regular
representation of the Hecke algebra $\mathbf{H}$ associated to the
Weyl group $W$ of $\mathfrak{g}$. 
In this picture, the Verma modules in $\mathcal{O}_0$ correspond to the
standard basis $\{H_w\,:\,w\in W\}$ of the Hecke algebra, the indecomposable 
projective modules (as well as the indecomposable projective functors) 
correspond to the  Kazhdan-Lusztig basis $\{\underline{H}_w\,:\,w\in W\}$, 
and the simple modules correspond to the dual Kazhdan-Lusztig basis 
$\{\underline{\hat{H}}_w\,:\,w\in W\}$. Consequently, the
Kazhdan-Lusztig combinatorics of  $\mathbf{H}$ becomes a very useful tool to study 
$\mathcal{O}_0$.

One particular representation-theoretic question for which 
Kazhdan-Lusztig combinatorics has shown to be very useful is 
the classical Kostant's problem for simple highest weight modules,
see \cite{Jo}. The original question asks whether the universal
enveloping algebra surjects onto the algebra of adjointly locally
finite endomorphisms of a module. In \cite{KMM}, a significant part
of this question was reduced to the study of the properties of 
the elements in $\mathbf{H}$ that have the form 
$\underline{\hat{H}}_x\underline{H}_y$. At the level of 
$\mathcal{O}_0$, these elements of $\mathbf{H}$ correspond to 
the outcome of applying the indecomposable projective functor
$\theta_y$ to the simple highest weight module $L_x$.

During our recent work on various aspects of Kostant's problem
for simple highest weight $\mathfrak{sl}_n$-modules in \cite{CM1,CM2}, we 
computed, using SageMath and GAP3, a lot of products of the form 
$\underline{\hat{H}}_x\underline{H}_y$ in small ranks in the case
of the Hecke algebra of the symmetric group $S_n$, with the 
outcome expressed as linear combinations of the elements of the
dual KL basis. Not directly related to the questions we studied
in \cite{CM1,CM2}, we made a bonus observation that, for fixed 
$x$ and $y$, the expression of $\underline{\hat{H}}_x\underline{H}_y$
seems to stabilize when the rank increases. Let us directly point 
out that  this is not an obvious phenomenon. The transformation 
matrix between the two bases changes very drastically with the
rank due to the fact that the duality between the two bases 
swaps the identity $e$ of $W$ with the longest element $w_0$.
And the element $w_0$ changes significantly with the rank.
In particular, the expression $\underline{\hat{H}}_x\underline{H}_y$
itself does depend on the rank, but this dependence seems to stabilize.
For comparison, the expression $\underline{{H}}_x\underline{H}_y$,
when written in the KL basis, is independent of the rank.

The main result of the present paper, Theorem~\ref{mainthm}, 
confirms this phenomenon: given two elements
$x,y\in S_n\subset S_k$, for $k\geq n$, we show that 
the expression $\underline{\hat{H}}_x\underline{H}_y$
with respect to the dual KL basis of $S_k$ stabilizes for $k\gg 0$.
As an application, in Theorem~\ref{thm77}, we show that one can
define an action of projective functors on the category of finite
length modules in the principal block of the BGG category
$\mathcal{O}$ for the algebra $\mathfrak{sl}_\infty$.

To prove the main result, we employ the Robinson-Schensted
correspondence for the symmetric group, \cite{Sch},
and the properties of the Kazhdan-Lusztig $\mu$-function 
listed in \cite{Wa}. For our application, we heavily use the properties 
of the restricted inverse limits of categories, described in \cite{EA15}.

The paper is organized as follows: In Section~\ref{s2} we collected
all preliminaries necessary for the formulation of our main result,
Theorem~\ref{mainthm}, which can be found in Subsection~\ref{s2.4}.
This main theorem is proved in Section~\ref{s3}. Finally,
Section~\ref{s4} details the application to projective functors
on category $\mathcal{O}$ for $\mathfrak{sl}_\infty$.
As a bonus, we prove an extension vanishing property for 
simple modules in the principal block of 
$\mathcal{O}$ for $\mathfrak{sl}_\infty$, answering (in this particular case)
a question from  \cite{CP}.
\vspace{5mm}

\subsection*{Acknowledgements}

The first author is partially supported by Vergstiftelsen.
The second author is partially supported by the Swedish Research Council.
\vspace{5mm}

\section{Hecke algebra of the symmetric group}\label{s2}

\subsection{Hecke algebra}\label{s2.1}

For $n\in\mathbb{Z}_{>0}$, consider the {\em symmetric group}
$S_n$ and its {\em Hecke algebra} $\mathbf{H}_n$. The latter 
is a $\mathbb{Z}[v,v^{-1}]$-algebra with {\em standard basis}
$\{H^{(n)}_w\,:\,w\in S_n\}$ and multiplication uniquely determined
by
\begin{equation}\label{eq0000}
H^{(n)}_{s} H^{(n)}_w=
\begin{cases}
H^{(n)}_{sw},& \text{ if } \ell(sw)>\ell(w);\\
H^{(n)}_{sw}+ (v^{-1}-v)H^{(n)}_w, & \text{ if } \ell(sw)<\ell(w);
\end{cases}
\end{equation}
where $s$ is an elementary transposition, $w\in S_n$
and $\ell$ denotes the length function. Note that
$H^{(n)}_e$ is the unit element of this algebra and that
$(H^{(n)}_s)^2=H^{(n)}_e+(v^{-1}-v)H^{(n)}_s$, for any elementary transposition $s$.
Consequently, all $H^{(n)}_s$, and hence all $H^{(n)}_w$ as well,
are invertible elements of $\mathbf{H}_n$.

We can view the algebra $\mathbf{H}_n$ as a subalgebra of $\mathbf{H}_{n+1}$
by sending $H^{(n)}_w$ to $H^{(n+1)}_w$, for $w\in S_n$.

\subsection{Kazhdan-Lusztig basis}\label{s2.2}

The {\em bar involution} on $\mathbf{H}_n$, denoted
$h\mapsto \overline{h}$, is  defined
by sending $v$ to $v^{-1}$ and $H^{(n)}_s$ to $(H^{(n)}_s)^{-1}$.
The algebra $\mathbf{H}_n$ has, see \cite{KL,So}, the
{\em Kazhdan-Lusztig basis} $\{\underline{H}^{(n)}_w\,:\,w\in S_n\}$ 
which is uniquely determined by the properties
\begin{displaymath}
\overline{\underline{H}^{(n)}_w}=\underline{H}^{(n)}_w
\qquad\text{ and }\qquad
\underline{H}^{(n)}_w= H^{(n)}_w+\sum_{x\in S_n}
v\mathbb{Z}[v]H^{(n)}_x.
\end{displaymath}

The entries of the transformation matrix between the
two bases are known as the {\em Kazhdan-Lusztig polynomials}
$\{p^{(n)}_{x,y}\,:\, x,y\in S_n\}$, that is, for $x,y\in S_n$, we have:
\begin{equation}\label{eq0001}
\underline{H}^{(n)}_y=\sum_{x\in S_n} p^{(n)}_{x,y}H^{(n)}_x.
\end{equation}
Note that  these polynomials do not really depend on $n$ in the
sense that, for $x,y\in S_n$, we have $p^{(n)}_{x,y}=p^{(n+1)}_{x,y}$.
Also, note that $p^{(n)}_{x,y}\neq 0$ implies that $x\leq y$ with respect to 
the Bruhat order on $S_n$. In particular, this implies that our realization of
$\mathbf{H}_n$ as a subalgebra of $\mathbf{H}_{n+1}$ is compatible with
the KL basis, i.e.  $\underline{H}^{(n)}_w=\underline{H}^{(n+1)}_w$, for $w\in S_n$.

\subsection{Dual Kazhdan-Lusztig basis}\label{s2.3}

We have a symmetric $\mathbb{Z}[v,v^{-1}]$-bilinear
form on $\mathbf{H}_n$ defined, for $x,y\in S_n$, via
\begin{displaymath}
(H^{(n)}_x,H^{(n)}_y)=\delta_{x,y}. 
\end{displaymath}
With respect to this form, we have the 
{\em  dual Kazhdan-Lusztig basis} 
$\{\hat{\underline{H}}^{(n)}_w\,:\,w\in S_n\}$
(see e.g. \cite[Subsection~7.1]{MS}), that is,
for $x,y\in S_n$, we have 
\begin{equation}\label{eqform}
({\underline{H}}^{(n)}_x,\hat{\underline{H}}^{(n)}_y)=\delta_{x,y}. 
\end{equation}

\subsection{Main result}\label{s2.4}

For $h\in \mathbf{H}_n$ and $w\in S_n$, we define 
$[h:\underline{\hat{H}}^{(n)}_w]\in \mathbb{Z}[v,v^{-1}]$ 
to be the coefficient of $\underline{\hat{H}}^{(n)}_w$ in  $h$ when expressed in the
dual Kazhdan-Lusztig basis, i.e.  
\begin{displaymath}
h=\sum_{w\in S_n} [h:\underline{\hat{H}}^{(n)}_w] \hat{\underline{H}}^{(n)}_w.
\end{displaymath}

\begin{theorem}\label{mainthm}
\begin{enumerate}[$($a$)$]
\item\label{mainthm.1}
For all $n\geq 1$, all $x,y,w\in S_n$ and all $k,m>n$, we have
\begin{displaymath}
[\hat{\underline{H}}^{(k)}_x{\underline{H}}^{(k)}_y:\underline{\hat{H}}^{(k)}_w]=
[\hat{\underline{H}}^{(m)}_x{\underline{H}}^{(m)}_y:\underline{\hat{H}}^{(m)}_w]. 
\end{displaymath}
\item\label{mainthm.2}
For any $n\geq 1$, all $x,y\in S_n$, all $k>2^{\ell(y)}n$ 
and all $w\in S_{k}$, we have that the inequality 
$[\hat{\underline{H}}^{(k)}_x{\underline{H}}^{(k)}_y:\underline{\hat{H}}^{(k)}_w]\neq 0$
implies $w\in S_{2^{\ell(y)}n}$.
\end{enumerate}
\end{theorem}

\subsection{Example}\label{s2.5}
Denote the transposition $(i,i+1)$ by $s_i$. Then we have:
\begin{displaymath}
\hat{\underline{H}}^{(4)}_{s_1s_2s_3s_1s_2s_1}{\underline{H}}^{(4)}_{s_3}=
\hat{\underline{H}}^{(4)}_{s_2s_1s_3s_2s_1}+
(v+v^{-1})\hat{\underline{H}}^{(4)}_{s_1s_2s_3s_1s_2s_1}
\end{displaymath}
and, for all $m>4$, we have 
\begin{displaymath}
\hat{\underline{H}}^{(m)}_{s_1s_2s_3s_1s_2s_1}{\underline{H}}^{(m)}_{s_3}=
\hat{\underline{H}}^{(m)}_{s_2s_1s_3s_2s_1}+
\hat{\underline{H}}^{(m)}_{s_1s_2s_3s_1s_2s_1s_4}+
(v+v^{-1})\hat{\underline{H}}^{(m)}_{s_1s_2s_3s_1s_2s_1}.
\end{displaymath}
The latter is easy to check, using GAP3, for $m\leq 8$.

\section{Proof of Theorem~\ref{mainthm}}\label{s3}

\subsection{Proof of Claim~\eqref{mainthm.1}}\label{s3.0}

Using \eqref{eqform}, we can interpret the value
$[\hat{\underline{H}}^{(k)}_x{\underline{H}}^{(k)}_y:\underline{\hat{H}}^{(k)}_w]$
as $(\hat{\underline{H}}^{(k)}_x{\underline{H}}^{(k)}_y,{\underline{H}}^{(k)}_w)$.
The latter equals 
$(\hat{\underline{H}}^{(k)}_x,{\underline{H}}^{(k)}_{y^{-1}}{\underline{H}}^{(k)}_w)$
and hence coincides with the multiplicity of ${\underline{H}}^{(k)}_x$
in ${\underline{H}}^{(k)}_{y^{-1}}{\underline{H}}^{(k)}_w$. Since the
canonical embeddings of the Hecke algebras are compatible with the KL basis, 
this multiplicity does not depend on $k$.

It remains to prove Claim~\eqref{mainthm.2}.

\subsection{Reduction to simple reflections}\label{s3.1}

As a first step, we claim that it is enough to prove Theorem~\ref{mainthm}
under the assumption that $y=s$ is a simple reflection.

For this, consider the structure constants of $\mathbf{H}_n$
with respect to the Kazhdan-Lusztig basis:
\begin{displaymath}
{\underline{H}}^{(n)}_a{\underline{H}}^{(n)}_b=
\sum_{c\in S_n}\gamma_{a,b}^c {\underline{H}}^{(n)}_c.
\end{displaymath}
Note that $\gamma_{a,b}^c\in\mathbb{Z}[v,v^{-1}]$ does not depend on $n$.
From \cite[Theorem~2]{So} it follows that
each $\gamma_{a,b}^c$ has
non-negative integer coefficients. Furthermore, if 
$y=s_1s_2\dots s_r$ is a reduced expression, then
${\underline{H}}^{(n)}_y$ appears with a non-zero coefficient
in the decomposition of the product
\begin{equation}\label{eq1}
{\underline{H}}^{(n)}_{s_1}{\underline{H}}^{(n)}_{s_2}
\dots{\underline{H}}^{(n)}_{s_r}. 
\end{equation}

If Theorem~\ref{mainthm} is proved under the assumption that 
$y=s$ is a simple reflection, then, using distributivity
of algebraic operations
and induction on $r$, Claim~\eqref{mainthm.2} of 
Theorem~\ref{mainthm} follows in the general case.

Therefore, in the remainder of this section, we will prove
Claim~\eqref{mainthm.2} of Theorem~\ref{mainthm} under the assumption that 
$y=s$ is a simple reflection.

\subsection{The $\mu$-function}\label{s3.2}

For $x\in S_n$ and a simple reflection $s$, recall that 
\begin{equation}\label{eq2}
\hat{\underline{H}}^{(n)}_x {\underline{H}}^{(n)}_{s}=
\begin{cases}
\displaystyle
(v+v^{-1})\hat{\underline{H}}^{(n)}_x +
\sum_{u\in S_n\,:\, \ell(us)>\ell(u)}\mu^{(n)}(x,u)\hat{\underline{H}}^{(n)}_u,
& \text{ if } \ell(xs)<\ell(x);\\
0,& \text{ if } \ell(xs)>\ell(x). 
\end{cases}
\end{equation}
Here $\mu^{(n)}(x,u)\in\mathbb{Z}_{\geq 0}$ is the value of the 
{\em Kazhdan-Lusztig $\mu$-function}, see \cite{KL}. Directly
from the definition of $\mu$ it follows that, for 
$x,u\in S_n\subset S_{n+1}$, we have 
$\mu^{(n)}(x,u)=\mu^{(n+1)}(x,u)$. Therefore, from now
on, we will drop the superscript $(n)$ for $\mu$.

We need to recall some basic properties of the $\mu$-function.
Denote by $\prec$ the Bruhat order on $S_n$.
For $w\in S_n$, denote by $\mathbf{lds}(w)$ and $\mathbf{rds}(w)$
the {\em left and right descent sets} of $w$, respectively.
This means that $\mathbf{lds}(w)$ consists of all simple reflections
$t\in S_n$ such that $tw\prec w$. Also, $\mathbf{rds}(w)$ consists 
of all simple reflections $t\in S_n$ such that $wt\prec w$.
The function $\mu$ has the following properties, see \cite{KL,BB,Wa}:
\begin{itemize}
\item $\mu(x,u)\neq 0$ only if $x\prec u$ or $u\prec x$;
\item $\mu(x,u)=\mu(u,x)$;
\item $\mu(x,u)=1$ if $x\prec u$ and $\ell(x)=\ell(u)-1$;
\item if $x\prec u$ and $\ell(x)<\ell(u)-1$, then $\mu(x,u)=0$
unless $\mathbf{lds}(u)\subset \mathbf{lds}(x)$ and
$\mathbf{rds}(u)\subset \mathbf{rds}(x)$.
\end{itemize}

\subsection{Cell combinatorics}\label{s3.3}

For two elements $a,b\in S_n$, we write $a\geq_L b$ provided that there
exists $c\in S_n$ such that ${\underline{H}}^{(n)}_a$ appears with a 
non-zero coefficient when expressing ${\underline{H}}^{(n)}_c{\underline{H}}^{(n)}_b$
in the KL-basis. This defines the so-called {\em left pre-order} on $S_n$ and the corresponding
equivalence classes ($a\sim_L b$ if and only if $a\geq_L b$ and $b\geq_L a$) are
called {\em left cells}. The {\em right pre-order} $\geq_R$ and the
{\em right cells} ($\sim_R$) are defined similarly using 
${\underline{H}}^{(n)}_b{\underline{H}}^{(n)}_c$.
The {\em two-sided pre-order} $\geq_J$ and the
{\em two-sided cells} ($\sim_J$) are defined similarly using 
${\underline{H}}^{(n)}_{c_1}{\underline{H}}^{(n)}_b{\underline{H}}^{(n)}_{c_2}$.
Note that the identity $e\in S_n$ is the minimum element with respect to all
three pre-orders while the longest element $w_0\in S_n$ is the maximum 
element with respect to all three pre-orders.

In \cite[Section~5]{KL}, one can find the following combinatorial description
of cells. Recall the classical {\em Robinson-Schensted correspondence} 
\begin{displaymath}
S_n\overset{1:1}{\leftrightarrow}
\coprod_{\lambda\vdash n}\mathbf{SYT}_\lambda\times \mathbf{SYT}_\lambda
\end{displaymath}
that provides a bijection between $S_n$ and pairs of {\em standard Young tableaux}
of the same shape (the shape itself is a partition of $n$). We denote
the map from the left hand side to the right hand side of the above 
by $\mathbf{RS}$ and write $\mathbf{RS}(w)=(P_w,Q_w)$, where $P_w$
is the {\em insertion tableau} and $Q_w$ is the {\em recording tableau}, see \cite{Sch,Sa}.
Then, for $a,b\in S_n$, we have
\begin{itemize}
\item $a\sim_L b$ if and only if $Q_a=Q_b$;
\item $a\sim_R b$ if and only if $P_a=P_b$;
\item $a\sim_J b$ if and only if the shape of $P_a$ is the same as the shape of $P_b$.
\end{itemize}
Consequently, for $a,b\in S_n\subset S_{n+1}$, we have 
$a\sim_L b$ in $S_n$ if and only if $a\sim_L b$ in $S_{n+1}$
and similarly for the right and two-sided cells. Hence there is no 
confusion that we do not use any superscripts in the notation.

One important property of the cell combinatorics is given by Lemma~13(a) of
\cite{MM1}: for $a,b,c\in S_n$, the fact that 
$[\hat{\underline{H}}^{(n)}_a{\underline{H}}^{(n)}_b:\hat{\underline{H}}^{(n)}_c]\neq 0$
implies $c\leq_R a$.

The rows of Young tableaux are numbered from top to bottom and 
the columns of Young tableaux are numbered from left to right.
It is important that the descent sets of $w$ can be directly read off 
from the tableaux $P_w$ and $Q_w$. We have, see \cite{Sa}:
\begin{itemize}
\item $s_i=(i,i+1)\in\mathbf{lds}(w)$ if and only if $i+1$ appears
in $P_w$ in a row with strictly higher number than the 
number of the row of $P_w$ containing $i$;
\item $s_i=(i,i+1)\in\mathbf{rds}(w)$ if and only if $i+1$ appears
in $Q_w$ in a row with strictly higher number than the number 
of  the row of $Q_w$ containing $i$.
\end{itemize}

Finally, we note that, for $w\in S_n\subset S_{n+1}$, the
$S_{n+1}$-insertion tableau for $w$ is obtained from 
the $S_{n}$-insertion tableau for $w$ by adding a new box with $n+1$ in the first row.
Similarly for the recording tableau.

In \cite{Ge}, it is shown that the order $\leq_J$ on the set of all
partitions of $n$ coincides with the (opposite of the) dominance order.
We note that no combinatorial interpretations of $\leq_L$ (or $\leq_R$)
are known. This is a major obstacle for many arguments. 
For some partial results, see \cite{HHS}.

\subsection{Stability for the $\mu$-function}\label{s3.4}

For $w\in S_n$, denote by $\mathbf{m}(w)$ the maximal positive integer $k$
such that $w(k)\neq k$ (we set $\mathbf{m}(e)=0$).
The main observation that we need for 
our proof of Theorem~\ref{mainthm} is the following:

\begin{lemma}\label{lem-2}
Let $a,b\in S_n$ be such that $\mu(a,b)\neq 0$ and
$\mathbf{rds}(a)\setminus \mathbf{rds}(b)\neq \varnothing$.
Then  $\mathbf{m}(b)\leq 2\mathbf{m}(a)$.
\end{lemma}

\begin{proof}
We know that $\mu(a,b)\neq 0$ implies $a\prec b$ or $b\prec a$. If
$b\prec a$, then the claim of the lemma is obvious. Therefore it remains to 
consider the case $a\prec b$. 

Assume $\ell(a)=\ell(b)-1$ and let $t\in \mathbf{rds}(a)\setminus \mathbf{rds}(b)$.
Assume, for contradiction, that $\mathbf{m}(b)> 2\mathbf{m}(a)$
(note that $\mathbf{m}(a)\neq 0$ as we assume that
$\mathbf{rds}(a)\neq \varnothing$ and hence, in particular,  $a\neq e$). Then
$b=as_{\mathbf{m}(b)-1}$ and $s_{\mathbf{m}(b)-1}$ commutes with all simple
reflections involved in $a$, in particular, with $t$.
Consequently, $t\in \mathbf{rds}(b)$, a contradiction.

Finally, assume $\ell(a)<\ell(b)-1$. Then we must have 
both $\mathbf{rds}(b)\subset \mathbf{rds}(a)$
and $\mathbf{lds}(b)\subset \mathbf{lds}(a)$.
We assume  $\mathbf{rds}(a)\setminus \mathbf{rds}(b)\neq \varnothing$.
Let $t\in \mathbf{rds}(a)\setminus \mathbf{rds}(b)$. Then, from Formula~\eqref{eq2},
we have that $\hat{\underline{H}}^{(k)}_b$ appears with a non-zero coefficient in
the decomposition of $\hat{\underline{H}}^{(k)}_a{\underline{H}}^{(k)}_t$ 
with respect to the dual KL basis. In particular, $b\leq_J a$.

Let $\lambda$ be the shape of $P_a$ and $\mu$ be the shape of $P_b$.
Then we have $\lambda=(\lambda_1,\lambda_2,\dots)$ and
$\mu=(\mu_1,\mu_2,\dots)$. Since $b\leq_J a$, the partition $\mu$ dominates
$\lambda$, in particular, $n-\mu_1\leq n-\lambda_1$. 

In order to prove that $\mathbf{m}(b)<2\mathbf{m}(a)$, we need to show that,
for any $i> 2\mathbf{m}(a)$, the box $i$ appears in the first row of
both $P_b$ and $Q_b$. We do it for $P_b$ and for $Q_b$ the argument is similar.
Assume that $i$ does not appear in the first row of $P_b$. Since
$\mathbf{lds}(b)\subset \mathbf{lds}(a)$ and $s_{i-1}\not\in \mathbf{lds}(a)$,
we have $s_{i-1}\not\in \mathbf{lds}(b)$. Hence $i-1$ must appear in $P_b$ 
in the same row as $i$ or in a row with higher number. Similarly for
$i-2$ and so on all the way down to $\mathbf{m}(a)$. However,
$\lambda$ has at most $\mathbf{m}(a)-1$ boxes combined in rows $2$, $3$ and so on.
Since $\mu$ dominates $\lambda$, we have that $\mu$ has 
at most $\mathbf{m}(a)-1$ boxes combined in rows $2$, $3$ and so on as well.
This means that we will eventually find an element in $\mathbf{lds}(b)$
which is not in $\mathbf{lds}(a)$, a contradiction. The claim of the lemma follows.
\end{proof}

\subsection{Completing the proof of Theorem~\ref{mainthm}\eqref{mainthm.2}}\label{s3.5}

Now we can prove Claim~\eqref{mainthm.2} of  Theorem~\ref{mainthm} under 
the assumption that $y=s$ is a simple reflection.
If $x=e$, then $\hat{\underline{H}}^{(k)}_e{\underline{H}}^{(k)}_s=0$, 
so it remains to consider the case $x\neq e$.
If we look at Formula~\eqref{eq2} and take  Lemma~\ref{lem-2} into account, we see that,
if we fix $x$, and hence also $\mathbf{m}(x)$, and let $n$ increase, then all 
non-trivial summands in Formula~\eqref{eq2} will correspond to $u$ such that
$\mathbf{m}(u)\leq 2\mathbf{m}(x)$. In other words, all these summands will be the same,
for $n$ big enough. This completes the proof.

\section{Application: category $\mathcal{O}$ for $\mathfrak{sl}_\infty$}\label{s4}

\subsection{Classical category $\mathcal{O}$ for $\mathfrak{sl}_n$}\label{s4.1}

For a Lie algebra $\mathfrak{g}$, we denote by $U(\mathfrak{g})$
its universal enveloping algebra.

For $n\geq 2$, consider the Lie algebra $\mathfrak{sl}_n$ over the complex numbers
and consider the standard triangular decomposition
\begin{displaymath}
\mathfrak{sl}_n=\mathfrak{n}^{(n)}_-\oplus \mathfrak{h}^{(n)}
\oplus \mathfrak{n}^{(n)}_+ 
\end{displaymath}
of this Lie algebra (here $\mathfrak{n}^{(n)}_-$ is the subalgebra of all strictly
lower triangular matrices, $\mathfrak{n}^{(n)}_+$ is the subalgebra of all strictly
upper triangular matrices, and $\mathfrak{h}^{(n)}$ is the Cartan subalgebra of
all traceless diagonal matrices). Associated to this triangular decomposition, we have the
BGG category $\mathcal{O}$, see \cite{BGG,Hu}, defined as the full 
subcategory of the category of all finitely generated 
$\mathfrak{sl}_n$-modules consisting of all 
$\mathfrak{h}$-semi-simple and $U(\mathfrak{n}_+)$-locally finite modules.

Consider the {\em principal block $\mathcal{O}_0$} of $\mathcal{O}$, defined
as the direct summand containing the trivial $\mathfrak{sl}_n$-module.
To emphasize dependence on $n$ (which we will need later), we will 
write $\mathcal{O}_0^{(n)}$. Simple objects in $\mathcal{O}_0^{(n)}$ are 
indexed naturally by the elements of the symmetric group $S_n$, the latter being
the Weyl group of $(\mathfrak{sl}_n,\mathfrak{h}^{(n)})$. We denote by 
$L_w^{(n)}$ the simple object corresponding to $w\in S_n$. As
an $\mathfrak{sl}_n$-module, $L_w^{(n)}$ is the simple highest weight
module with highest weight $w\cdot 0$, where $\cdot$ is the dot-action of
$S_n$ on $(\mathfrak{h}^{(n)})^*$.

The category $\mathcal{O}_0^{(n)}$ carries the natural action of the monoidal
category $\mathscr{P}^{(n)}$ of {\em projective endofunctors},
see \cite{BG}. A projective functor is defined as a summand of the functor
of tensoring with a finite dimensional $\mathfrak{sl}_n$-module.
Indecomposable projective functors are in bijection with the elements of
$S_n$, we write $w\mapsto \theta_w^{(n)}$, where $\theta_w^{(n)}$ is uniquely 
determined by the property that it maps the indecomposable projective
cover of $L_e^{(n)}$ to the indecomposable projective cover of $L_w^{(n)}$.

The split Grothendieck ring $[\mathscr{P}^{(n)}]_{\oplus}$ acts on the 
Grothendieck group of $\mathcal{O}_0^{(n)}$ and this gives rise to a
right regular representation of the integral group algebra $\mathbb{Z}[S_n]$.

The category $\mathcal{O}_0^{(n)}$ is equivalent to the module category of
a finite dimensional associative algebra $A^{(n)}$ which is known to be Koszul,
see \cite{So0}. With respect to the corresponding $\mathbb{Z}$-grading,
both $\mathcal{O}_0^{(n)}$ and $\mathscr{P}^{(n)}$ admit (unique
up to equivalence) graded lifts,
${}^{\mathbb{Z}}\mathcal{O}_0^{(n)}$ and ${}^{\mathbb{Z}}\mathscr{P}^{(n)}$, 
respectively. Interpreting the corresponding Grothendieck groups as
$\mathbb{Z}[v,v^{-1}]$-modules, where $v$ acts as the shift of grading,
the decategorification of the action of ${}^{\mathbb{Z}}\mathscr{P}^{(n)}$
on ${}^{\mathbb{Z}}\mathcal{O}_0^{(n)}$ gives the right regular
$\mathbf{H}$-module, see \cite{St0}.

Under these identifications, $\theta_w^{(n)}$ corresponds to
$\underline{H}_w^{(n)}$ and $L_w^{(n)}$ to $\hat{\underline{H}}_w^{(n)}$,
so that we can directly see the relevance of the formula in Theorem~\ref{mainthm}
in this context: this formula combinatorially describes the outcome of the
action of  $\theta_y^{(k)}$ on $L_x^{(k)}$.

\subsection{Projective limit}\label{s4.2}

Fix now $n\in\mathbb{Z}_{>0}$ such that $n>2$ and consider the 
standard embedding $\mathfrak{sl}_{n-1}\subset \mathfrak{sl}_n$
with respect to the left upper corner of a matrix. Note that we have
$\mathfrak{sl}_{n-1}\cap \mathfrak{n}^{(n)}_-=
\mathfrak{n}^{(n-1)}_-$, $\mathfrak{sl}_{n-1}\cap \mathfrak{n}^{(n)}_+=
\mathfrak{n}^{(n-1)}_+$ and $\mathfrak{sl}_{n-1}\cap \mathfrak{h}^{(n)}=
\mathfrak{h}^{(n-1)}$.

Let $\mathbf{R}^{(n)}\subset (\mathfrak{h}^{(n)})^*$ be the root
system of $(\mathfrak{sl}_n,\mathfrak{h}^{(n)})$. Let
$\mathbb{Z}\mathbf{R}^{(n)}\subset (\mathfrak{h}^{(n)})^*$
be the additive subgroup generated by $\mathbf{R}^{(n)}$.
Then, for each $M\in \mathcal{O}_0^{(n)}$, we have
\begin{displaymath}
M\cong\bigoplus_{\lambda\in \mathbb{Z}\mathbf{R}^{(n)}} M_\lambda,
\quad\text{ where }\quad
M_\lambda=\{v\in M\,: hv=\lambda(h)v\text{ for all }h\in\mathfrak{h}^{(n)}\}.
\end{displaymath}
Let  $\mathbf{Q}^{(n-1)}\subset \mathbf{R}^{(n)}$ be the subset of all
roots corresponding to the elements of $\mathfrak{sl}_{n-1}$
and $\mathbb{Z}\mathbf{Q}^{(n-1)}\subset (\mathfrak{h}^{(n)})^*$
be the corresponding additive subgroup. Then, for any 
$M\in \mathcal{O}_0^{(n)}$, the space
\begin{displaymath}
\mathbf{F}^n_{n-1}M:=
\bigoplus_{\lambda\in \mathbb{Z}\mathbf{Q}^{(n-1)}} M_\lambda
\end{displaymath}
is invariant under the $\mathfrak{sl}_{n-1}$-action and is, in fact,
an object of $\mathcal{O}_0^{(n-1)}$.
By restriction, this defines a functor
\begin{displaymath}
\mathbf{F}^n_{n-1}:\mathcal{O}_0^{(n)}\to\mathcal{O}_0^{(n-1)}
\end{displaymath}
and hence gives rise to an inverse system of categories and functors
\begin{equation}\label{eqn1}
\xymatrix{
\mathcal{O}_0^{(2)}&&
\mathcal{O}_0^{(3)}\ar[ll]_{\mathbf{F}^3_2}
&&
\mathcal{O}_0^{(4)}\ar[ll]_{\mathbf{F}^4_3}&&
\dots\ar[ll]_{\mathbf{F}^5_4}
}
\end{equation}

We note that there is an alternative way to think of 
$\mathbf{F}^n_{n-1}$, see \cite{CMZ}.
Consider the usual embedding $S_{n-1}\subset S_n$ and recall that 
the simple objects of $\mathcal{O}_0^{(n)}$ are $L_w^{(n)}$, where $w\in S_n$.
Let $\mathcal{X}_n$ be the Serre subcategory of $\mathcal{O}_0^{(n)}$
generated by all simples $L_w$, where $w\in S_n\setminus S_{n-1}$.
Then $\mathbf{F}^n_{n-1}$ induces and equivalence between
$\mathcal{O}_0^{(n)}/\mathcal{X}_n$ and $\mathcal{O}_0^{(n-1)}$ and this 
equivalence sends  $L_w^{(n)}$ to $L_w^{(n-1)}$, for $w\in S_{n-1}$. 
This  alternative description,  in particular, implies that 
\eqref{eqn1} admits a graded lift
\begin{equation}\label{eqn2}
\xymatrix{
{}^{\mathbb{Z}}\mathcal{O}_0^{(2)}&&
{}^{\mathbb{Z}}\mathcal{O}_0^{(3)}\ar[ll]_{{}^{\mathbb{Z}}\mathbf{F}^3_2}
&&
{}^{\mathbb{Z}}\mathcal{O}_0^{(4)}\ar[ll]_{{}^{\mathbb{Z}}\mathbf{F}^4_3}&&
\dots\ar[ll]_{{}^{\mathbb{Z}}\mathbf{F}^5_4}
}
\end{equation}
We define the category $\mathcal{C}$ as the limit of \eqref{eqn1}
and the category ${}^{\mathbb{Z}}\mathcal{C}$ as the limit of \eqref{eqn2},
see e.g. \cite[Section~3]{EA15}.

Note that all categories $\mathcal{O}_0^{(n)}$ are length categories and
the functors $\mathbf{F}^n_{n-1}$ are exact and send simple
objects to either simple objects or zero. Therefore they are 
``shortening'' in the terminology of \cite[Section~4]{EA15},
so that we can consider the {\em restricted limits}  of the above 
inverse systems, which we denote by $\mathcal{C}^{\mathrm{res}}$
and ${}^{\mathbb{Z}}\mathcal{C}^{\mathrm{res}}$, respectively.
We will provide more details on these later, in Subsection~\ref{s4.5}.
An important property of both $\mathcal{C}^{\mathrm{res}}$ 
and ${}^{\mathbb{Z}}\mathcal{C}^{\mathrm{res}}$ is that 
they are length categories, see \cite[Lemma~4.1.3]{EA15}.

\subsection{Simple highest weight modules for $\mathfrak{sl}_\infty$}\label{s4.3}
This subsection is inspired by \cite{Na1,Na2,Na3,CP}.

Consider the directed system
$\mathfrak{sl}_2\hookrightarrow 
\mathfrak{sl}_3\hookrightarrow \dots$
given by inclusions with respect to the left upper corner 
and denote by $\mathfrak{sl}_\infty$ the Lie algebra that is the limit of this system.
Taking the component-wise limit, gives a triangular decomposition
\begin{displaymath}
\mathfrak{sl}_\infty=\mathfrak{n}^{(\infty)}_-\oplus \mathfrak{h}^{(\infty)}
\oplus \mathfrak{n}^{(\infty)}_+. 
\end{displaymath}
In this setup, one can define a few variations of the 
classical category $\mathcal{O}$, see e.g. \cite{Na1,Na2,Na3,CP}.
We will restrict our attention to the following category.

For $\lambda\in (\mathfrak{h}^{(\infty)})^*$, consider the
{\em Verma module} $M(\lambda)^{(\infty)}$ defined as the 
quotient of $U(\mathfrak{sl}_\infty)$ by the left ideal
generated by $\mathfrak{n}^{(\infty)}_+$ and all $h-\lambda(h)$,
where $h\in \mathfrak{h}^{(\infty)}$ (note that this ideal
is not finitely generated). The module $M(\lambda)^{(\infty)}$
is an $\mathfrak{h}^{(\infty)}$-weight module with finite
dimensional weight spaces. As usual, the Verma module
$M(\lambda)^{(\infty)}$ has simple top, denoted 
$L(\lambda)^{(\infty)}$. The module $L(\lambda)^{(\infty)}$ 
is the simple highest weight module with highest weight $\lambda$. 

Consider the full subcategory $\mathcal{O}^{(\infty,\mathrm{fl})}$ of the category
of all weight $\mathfrak{h}^{(\infty)}$-weight $\mathfrak{sl}_\infty$-modules
which consists of all objects that have a finite composition series
with simple highest weight subquotients. Note that the analogue
of $\mathcal{O}^{(\infty,\mathrm{fl})}$ for each
$\mathfrak{sl}_n$, with $n<\infty$, is the usual BGG category $\mathcal{O}$.
We denote by $\mathcal{O}^{(\infty,\mathrm{fl})}_0$ the direct summand of
$\mathcal{O}^{(\infty,\mathrm{fl})}$ which contains the trivial
$\mathfrak{sl}_\infty$-module.

We emphasize that, in general, the module $M(\lambda)^{(\infty)}$ has infinite 
length and hence does not belong to $\mathcal{O}^{(\infty,\mathrm{fl})}$.
For example, $M(0)^{(\infty)}$ does not belong to 
$\mathcal{O}^{(\infty,\mathrm{fl})}$.

Denote by $S_\infty$ the limit of the directed system
$S_1\hookrightarrow S_2\hookrightarrow \dots$ defined
via the obvious inclusions. The group $S_\infty$ acts 
on $(\mathfrak{h}^{(\infty)})^*$ both in the obvious way
and via the dot-action.

For $n\geq 2$, let $\mathbf{Q}^{(n)}\subset (\mathfrak{h}^{(\infty)})^*$
be the set of all roots of the $\mathfrak{sl}_n$-part of 
$\mathfrak{sl}_\infty$. Let $\mathbb{Z}\mathbf{Q}^{(n)}$
be the additive subgroup of $(\mathfrak{h}^{(\infty)})^*$ 
spanned by $\mathbf{Q}^{(n)}$. For 
$M\in \mathcal{O}^{(\infty,\mathrm{fl})}_0$
consider the $\mathfrak{sl}_n$-module
\begin{equation}\label{eq72}
\mathbf{F}^\infty_n M:=
\bigoplus_{\mu\in \mathbb{Z}\mathbf{Q}^{(n)}} M_\mu.
\end{equation}
Then $\mathbf{F}^\infty_n$ gives rise to a functor from
$\mathcal{O}^{(\infty,\mathrm{fl})}_0$ to $\mathcal{O}^{(n)}_0$
which sends $L(w\cdot 0)^{(\infty)}$ to $L_w^{(n)}$,
if $w\in S_n$, and to zero otherwise.

\begin{proposition}\label{prop71}
For $\lambda\in (\mathfrak{h}^{(\infty)})^*$, we have
$L(\lambda)^{(\infty)}\in \mathcal{O}^{(\infty,\mathrm{fl})}_0$
if and only if $\lambda\in S_\infty\cdot 0$.
\end{proposition}

\begin{proof}
The ``only if'' part follows from \cite[Theorem~2.20]{Na2}. 

Let $w\in S_\infty$ and $s$ be a simple reflection such that $sw\succ w$. 
To prove the ``if'' part, it is enough to show that
$\mathrm{Ext}^1_{\mathcal{O}^{(\infty)}_0}(L(w\cdot 0)^{(\infty)},
L(sw\cdot 0)^{(\infty)})\neq 0$. For this, consider the Verma module
$M(w\cdot 0)^{(\infty)}$. Since $s$ is a simple reflection
and $sw\succ w$, the dimension of $M(w\cdot 0)^{(\infty)}_{sw\cdot 0}$ equals $1$.

Let $N$ be the biggest submodule of $M(w\cdot 0)^{(\infty)}$ with the
property $N_{w\cdot 0}=N_{sw\cdot 0}=0$. Then $M(w\cdot 0)^{(\infty)}/N$
is an indecomposable module with simple top $L(w\cdot 0)^{(\infty)}$
and simple socle $L(ws\cdot 0)^{(\infty)}$. If $M(w\cdot 0)^{(\infty)}/N$
has length $2$, we are done. If not, $M(w\cdot 0)^{(\infty)}/N$
has some simple subquotient of the form $L(x\cdot 0)^{(\infty)}$,
for some $x\in S_\infty$ such that $x\neq w$ and $x\neq sw$. 

Choose $n$ such that  $w,sw,x\in S_n$ and 
consider $K:=\mathbf{F}^\infty_n(M(w\cdot 0)^{(\infty)}/N)$.
This module is, by construction, a quotient of the Verma module $M(w\cdot 0)^{(n)}$.
As the space $\mathrm{Ext}^1_{\mathcal{O}^{(n)}_0}(L(w\cdot 0)^{(n)},
L(sw\cdot 0)^{(n)})$ is one-dimensional, the module
$M(w\cdot 0)^{(n)}$ has a length $2$ quotient with simple top
$L(w\cdot 0)^{(n)}$ and simple socle $L(sw\cdot 0)^{(n)}$.
Consequently, any $\mathfrak{sl}_n$-submodule of 
$K$ generated by any $L(x\cdot 0)^{(\infty)}$ as above
does not contain any subquotient of the form $L(sw\cdot 0)^{(n)}$.
Since this is true for any $x$ as above and any $n\gg 0$, we get a 
contradiction with the fact that $L(ws\cdot 0)^{(\infty)}$ is the simple 
socle of $M(w\cdot 0)^{(\infty)}/N$. Therefore our assumption that 
$M(w\cdot 0)^{(\infty)}/N$ has length different from $2$ is false. The claim of the
proposition follows.
\end{proof}

% \subsection{Direct limit of projective functors}\label{s4.4}
% 
% We want to construct a limit of the monoidal categories
% $\mathscr{P}^{(n)}$. For this, let us recall 
% {\em Soergel's combinatorial description} of 
% $\mathscr{P}^{(n)}$ from \cite{So92}.
% 
% Consider the coinvarinat algebra $\mathtt{C}^{(n)}$ of 
% the Weyl group $S_n$. For a simple reflection $s$,
% let $(\mathtt{C}^{(n)})^s$ be the algebra of all $s$-invariants
% in $\mathtt{C}^{(n)}$. Denote by $\mathscr{S}^{(n)}$
% the strictly full monoidal subcategory of
% $\mathtt{C}^{(n)}$-mod-$\mathtt{C}^{(n)}$ generated by
% $\mathtt{C}^{(n)}\otimes_{(\mathtt{C}^{(n)})^s}\mathtt{C}^{(n)}$,
% where $s$ runs through the set of all simple reflections in
% $S_n$. Then $\mathscr{S}^{(n)}$ is monoidally equivalent to 
% $\mathscr{P}^{(n)}$.
% 

\subsection{Action of projective functors on restricted limits}\label{s4.5}

For each $w\in S_\infty$, we want to define an endofunctor 
$\theta_w^{(\infty)}$ of the category $\mathcal{C}$
(and then of the category $\mathcal{C}^{\mathrm{res}}$ as well). For this, recall
from \cite[Definition~3.1.1]{EA15} that objects of $\mathcal{C}$
are pairs $(\{C_i\,:\,i\geq 2\},\{\phi_i\,:\, i\geq 3\})$, where
$C_i\in\mathcal{O}_0^{(i)}$ and $\phi_i:\mathbf{F}^i_{i-1}C_i
\overset{\cong}{\longrightarrow} C_{i-1}$
is an isomorphism. A morphism $f$ from an object
$(\{C_i\,:\,i\geq 2\},\{\phi_i\,:\, i\geq 3\})$ to
$(\{C'_i\,:\,i\geq 2\},\{\phi'_i\,:\, i\geq 3\})$
consists of $f_i:C_i\to C'_i$, for all $i\geq 2$, such that 
the following diagram commutes:
\begin{displaymath}
\xymatrix{
\mathbf{F}^i_{i-1}C_i\ar[rr]^{\phi_i}\ar[d]_{\mathbf{F}^i_{i-1}f_i}
&&C_{i-1}\ar[d]^{f_{i-1}}\\
\mathbf{F}^i_{i-1}C'_i\ar[rr]^{\phi'_i}&&C'_{i-1}
}
\end{displaymath}
Composition of morphisms is component-wise and the identity morphisms are 
given by the identities.

If $w$ is the identity, we just define $\theta_e^{\infty}$
as the identity endofunctor of $\mathcal{C}$. If $w\neq e$,
let $k=\mathbf{m}(w)$. By \cite[Theorem~37]{CMZ}, for all
$i>k$, we can fix isomorphisms 
$\alpha_i:\mathbf{F}^i_{i-1}\circ\theta_w^{(i)}\cong 
\theta_w^{(i-1)}\circ\mathbf{F}^i_{i-1}$.

We define the action of $\theta_w^{(\infty)}$
on the object $(\{C_i\,:\,i\geq 2\},\{\phi_i\,:\, i\geq 3\})$
as the object $(\{D_i\,:\,i\geq 2\},\{\psi_i\,:\, i\geq 3\})$, where
\begin{displaymath}
D_i=
\begin{cases}
\theta_w^{(i)}C_i,& i\geq k;\\
\mathbf{F}^{i+1}_{i}\circ\dots\circ
\mathbf{F}^{k-1}_{k-2}\circ\mathbf{F}^{k}_{k-1}\theta_w^{(k)}C_{k},& i<k;
\end{cases}
\end{displaymath}
and
\begin{displaymath}
\psi_i=
\begin{cases}
\theta_w^{(i-1)}(\psi_{i-1})\circ (\alpha_i)_{C_i},& i> k;\\
\mathrm{id},& i\leq k.
\end{cases}
\end{displaymath}
The action of $\theta_w^{(\infty)}$ on a morphism $f=\{f_i\}$ is defined
as $g=\{g_i\}$, where
\begin{displaymath}
g_i=
\begin{cases}
\theta_w^{(i-1)}(f_i)& i>k;\\
\mathbf{F}^{i+1}_{i}\circ\dots\circ
\mathbf{F}^{k-1}_{k-2}\circ\mathbf{F}^{k}_{k-1}\theta_w^{(k)}(f_{k+1}),& i\leq k.
\end{cases}
\end{displaymath}

\begin{lemma}\label{lem74}
The above gives  a well-defined endofunctor
of $\mathcal{C}$. The functor $\theta_w^{(\infty)}$ is exact and
biadjoint to $\theta_{w^{-1}}^{(\infty)}$.
\end{lemma}

\begin{proof}
This follows directly from the construction and definitions 
since each $\theta_w^{(i)}$ is exact and
biadjoint to $\theta_{w^{-1}}^{(i)}$,
see also \cite[Proposition~3.2.4]{EA15}.
\end{proof}

\begin{proposition}\label{prop75}
For any $w\in S_\infty$, the functor $\theta_w^{(\infty)}$ 
restricts to an endofunctor of the category $\mathcal{C}^{\mathrm{res}}$.
\end{proposition}

\begin{proof}
Since the functor $\theta_w^{(\infty)}$ is exact, we just need 
to prove that, for a fixed $u\in S_\infty$, the length of the module 
$\theta_w^{(i)}L_u^{(i)}$ stabilizes, for all $i\gg 0$. This
follows directly from Theorem~\ref{mainthm} and the 
categorification remarks in Subsection~\ref{s4.1}.
\end{proof}

\subsection{The Hecke algebra for $S_\infty$}\label{s4.6}

Consider the {\em Hecke algebra} $\mathbf{H}_\infty$,
defined as the $\mathbb{Z}[v,v^{-1}]$-algebra with {\em standard basis}
$\{H^{(\infty)}_w\,:\,w\in S_\infty\}$ and multiplication uniquely defined
by \eqref{eq0000} (with the superscript $(n)$ replaced by $(\infty)$). We have the
{\em Kazhdan-Lusztig basis} $\{\underline{H}^{(\infty)}_w\,:\,w\in S_\infty\}$
of $\mathbf{H}_\infty$, defined by \eqref{eq0001} 
(with the superscript $(n)$ replaced by $(\infty)$, here it is important that
the KL polynomials do not depend on $n$ at all).

Consider the free $\mathbb{Z}[v,v^{-1}]$-module $\hat{\mathbf{H}}_\infty$
with the basis $\{\hat{\underline{H}}^{(\infty)}_w\,:\,w\in S_\infty\}$.
For $x,y,w\in S_\infty$, define 
$[\hat{\underline{H}}^{(\infty)}_x{\underline{H}}^{(\infty)}_y:\hat{\underline{H}}^{(\infty)}_w]$
as the stabilized value of the expression 
$[\hat{\underline{H}}^{(k)}_x{\underline{H}}^{(k)}_y:\hat{\underline{H}}^{(k)}_w]$
in Theorem~\ref{mainthm}, where $k\gg 0$.
From Theorem~\ref{mainthm}, it follows that, setting, for 
arbitrary $x,y\in S_\infty$, 
\begin{displaymath}
\hat{\underline{H}}^{(\infty)}_x
\underline{H}^{(\infty)}_y:=
\sum_{w\in S_\infty}
[\hat{\underline{H}}^{(\infty)}_x{\underline{H}}^{(\infty)}_y:\hat{\underline{H}}^{(\infty)}_w]
\hat{\underline{H}}^{(\infty)}_w,
\end{displaymath}
turns $\hat{\mathbf{H}}_\infty$ into a (right)
$\mathbf{H}_\infty$-module.

\begin{remark}\label{rem097}
{\em
Unlike the finite case, $\hat{\mathbf{H}}_\infty$ does not coincide
with the right regular $\mathbf{H}_\infty$-mo\-du\-le.
However, one can identify $\hat{\mathbf{H}}_\infty$ with a submodule 
of a certain enlargement of the right regular $\mathbf{H}_\infty$-module.
For this, we need to allow formal Laurent series as coefficients in front
of the elements in the standard  (or Kazhdan-Lusztig) basis. Then 
one can take the limit, where $n$ goes to infinity, of the expressions
of the elements of the dual Kazhdan Lusztig basis in terms of 
the standard  (or Kazhdan-Lusztig) basis. This will output an
infinite sum with formal Laurent series as coefficients.
}
\end{remark}

\subsection{Projective functors on 
$\mathcal{O}^{(\infty,\mathrm{fl})}_0$}\label{s4.7}

Our next observation is the following.

\begin{theorem}\label{thm77}
The categories $\mathcal{C}^{\mathrm{res}}$
and $\mathcal{O}^{(\infty,\mathrm{fl})}_0$ are equivalent.
\end{theorem}

\begin{proof}
Directly from the construction, for every $2\leq m< n$, we have 
$\mathbf{F}^n_m\circ\mathbf{F}^\infty_n=\mathbf{F}^\infty_m$.
Therefore, by the universal properties of limits, we have 
a faithful functor $\Phi$ from $\mathcal{O}^{(\infty)}_0$ to
$\mathcal{C}$ sending $M$ to 
$(\{\mathbf{F}^\infty_i M\},\{\mathrm{id}_{\mathbf{F}^\infty_i M}\})$ and with 
the obvious restriction action on morphisms.
This restricts to a faithful functor $\Phi^{\mathrm{res}}$ 
from $\mathcal{O}^{(\infty,\mathrm{fl})}_0$ 
to $\mathcal{C}^{\mathrm{res}}$ and we claim that this
latter functor is, in fact, an equivalence.

Given $(\{C_i\,:\,i\geq 2\},\{\phi_i\,:\, i\geq 3\})$ in
$\mathcal{C}^{\mathrm{res}}$, let $M$ be the vector space consisting
of the equivalence classes of all sequences $(v_k,v_{k+1},\dots)$,
where $k\geq 2$ and $v_{i-1}=\phi_i(v_{i})$, for all $i>k$. 
Two such sequences $(v_k,v_{k+1},\dots)$
and $(w_m,w_{m+1},\dots)$ are equivalent if there is $r>\max(k,m)$
such that $v_i=w_i$, for all $i>r$. The vector space structure
is defined component-wise on those components for which 
it makes sense (this means that, when adding $(v_k,v_{k+1},\dots)$
and $(w_m,w_{m+1},\dots)$, the indexing of the outcome will start
from $\max(k,m)$). Furthermore, the vector space $M$ carries the
natural structure of an $\mathfrak{sl}_\infty$-module defined again
by acting on those components for which the action makes sense
(that is, $u\in U(\mathfrak{sl}_m)$ acts on $(v_k,v_{k+1},\dots)$
outputting a vector whose indexing starts with $\max(k,m)$).
In fact, with the obvious action on morphisms, this defines a 
functor $\Psi$ from  $\mathcal{C}^{\mathrm{res}}$ to
$\mathcal{O}^{(\infty,\mathrm{fl})}_0$.

It follows directly from the definitions that 
$\Phi^{\mathrm{res}}\circ\Psi$ is isomorphic to the identity
endofunctor of $\mathcal{C}^{\mathrm{res}}$,
implying that $\Phi^{\mathrm{res}}$ is both dense and full.
The claim of the theorem follows.
\end{proof}

Combining Proposition~\ref{prop75} with Theorem~\ref{thm77},
for every $w\in S_\infty$, we have an 
endofunctor $\theta_w^{(\infty)}$ of 
$\mathcal{O}^{(\infty,\mathrm{fl})}_0$.
Putting all $\theta_w^{(\infty)}$ together, we have a 
monoidal category $\mathscr{P}^{(\infty)}$ of 
projective endofunctors of $\mathcal{O}^{(\infty,\mathrm{fl})}_0$.
The objects are the functors from the additive closure of all
$\theta_w^{(\infty)}$ and the morphisms are natural transformations
of functors. The monoidal product is composition of functors.
The split Grothendieck ring of $\mathscr{P}^{(\infty)}$ is 
isomorphic to $\mathbb{Z}[S_\infty]$ (resp. to the Hecke algebra
of $S_\infty$ in the graded version), where $\theta_w^{(\infty)}$
corresponds to $\underline{H}_w$.

The Grothendieck group of $\mathcal{O}^{(\infty,\mathrm{fl})}_0$
is a free abelian group (resp. a free $\mathbb{Z}[v,v^{-1}]$-module
in the graded version) with basis $[L_w^{\infty}]$, where $w\in S_n$.

\begin{corollary}\label{cor79}
The action of projective functors on the graded version
${}^\mathbb{Z}\mathcal{C}^{\mathrm{res}}$ of 
$\mathcal{O}^{(\infty,\mathrm{fl})}_0$ decategorifies to 
the  $\mathbf{H}_\infty$-module $\hat{\mathbf{H}}_\infty$,
where $[L_w^{\infty}]$ corresponds to $\underline{\hat{H}}_w$.
\end{corollary}

\begin{proof}
This follows from the definitions and  Theorem~\ref{mainthm}.
\end{proof}

\subsection{Koszulity of 
${}^\mathbb{Z}\mathcal{C}^{\mathrm{res}}$}\label{s4.9}

The following answers \cite[Question~5.3.3]{CP} in the case of
the category ${}^\mathbb{Z}\mathcal{C}^{\mathrm{res}}$.
As usual, $\mathrm{ext}$ denotes homogeneous extensions of degere zero
and $\langle 1\rangle$ is the degree shift that sends degree $0$
to degree $-1$.

\begin{theorem}\label{thm-cp}
For $x,y\in S_\infty$ and $i,j\in\mathbb{Z}_{\geq  0}$, 
the inequality $i\neq j$ implies
\begin{displaymath}
\mathrm{ext}^{i}_{\mathcal{C}^{\mathrm{res}}}(L(x\cdot 0)^{(\infty)},L(y\cdot 0)^{(\infty)}\langle -j\rangle)=0. 
\end{displaymath}
\end{theorem}

\begin{proof}
For brevity, we write $L_x$ and $L_y$ for 
$L(x\cdot 0)^{(\infty)}$ and $L(y\cdot 0)^{(\infty)}$, respectively.
We are going to prove the necessary extension vanishing using the 
interpretation of extensions as equivalence classes of exact sequences.
 
Consider two exact sequences 
\begin{equation}\label{eq-cp1}
\xymatrix{
0\ar[r]&L_y\langle -j\rangle\ar[r]&X_1\ar[r]&X_2\ar[r]&\dots\ar[r]& X_i\ar[r]& L_x\ar[r]& 0
}
\end{equation}
and
\begin{equation}\label{eq-cp2}
\xymatrix{
0\ar[r]&L_y\langle -j\rangle\ar[r]&Y_1\ar[r]&Y_2\ar[r]&\dots\ar[r]& Y_i\ar[r]& L_x\ar[r]& 0
}
\end{equation}
in ${}^\mathbb{Z}\mathcal{C}^{\mathrm{res}}$.
Choose $k$ such that the indices of all simple subquotients of all modules involved
in these two sequences belong to $S_k$. Such $k$ exists as all $X_s$ and all $Y_t$
involved have finite length. Then, applying the exact functor $\mathbf{F}^\infty_k$
to \eqref{eq-cp1} and \eqref{eq-cp2} gives two exact sequences
\begin{displaymath}
\xymatrix@C=7mm{
0\ar[r]&\mathbf{F}^\infty_kL_y\langle -j\rangle\ar[r]&
\mathbf{F}^\infty_kX_1\ar[r]&\dots\ar[r]& 
\mathbf{F}^\infty_kX_i\ar[r]& 
\mathbf{F}^\infty_kL_x\ar[r]& 0
}
\end{displaymath}
and
\begin{displaymath}
\xymatrix@C=7mm{
0\ar[r]&\mathbf{F}^\infty_kL_y\langle -j\rangle\ar[r]&
\mathbf{F}^\infty_kY_1\ar[r]&\dots\ar[r]& 
\mathbf{F}^\infty_kY_i\ar[r]& 
\mathbf{F}^\infty_kL_x\ar[r]& 0.
}
\end{displaymath}
As $\mathcal{O}_0^{(k)}$ is Koszul, see \cite{So0},
we know that 
$\mathrm{ext}^{i}_{\mathcal{O}_0^{(k)}}(L_x^{(k)},L_y^{(k)}\langle -j\rangle)=0$.
Therefore, by \cite[Theoreme~1, A X.121]{Bo}, there exists a commutative diagram
\begin{equation}\label{eq-cp5}
\xymatrix@C=7mm@R=7mm{
0\ar[r]&\mathbf{F}^\infty_kL_y\langle -j\rangle\ar[r]\ar@{=}[d]&
\mathbf{F}^\infty_kY_1\ar[r]\ar[d]&\dots\ar[r]& 
\mathbf{F}^\infty_kY_i\ar[r]\ar[d]& 
\mathbf{F}^\infty_kL_x\ar[r]\ar@{=}[d]& 0\\
0\ar[r]&\mathbf{F}^\infty_kL_y\langle -j\rangle\ar[r]&
Z_1\ar[r]&\dots\ar[r]& 
Z_i\ar[r]& 
\mathbf{F}^\infty_kL_x\ar[r]& 0\\
0\ar[r]&\mathbf{F}^\infty_kL_y\langle -j\rangle\ar[r]\ar@{=}[u]&
\mathbf{F}^\infty_kX_1\ar[r]\ar[u]&\dots\ar[r]& 
\mathbf{F}^\infty_kX_i\ar[r]\ar[u]& 
\mathbf{F}^\infty_kL_x\ar[r]\ar@{=}[u]& 0\\
}
\end{equation}
over ${}^{\mathbb{Z}}\mathcal{O}_0^{(k)}$ with exact rows. Applying to 
\eqref{eq-cp5} the exact parabolic induction $\mathbf{I}_k^{\infty}$ from $\mathfrak{sl}_k$ to
$\mathfrak{sl}_{\infty}$ (this induction is, as usual, left adjoint to 
$\mathbf{F}_k^{\infty}$) outputs a commutative diagram
\begin{equation}\label{eq-cp6}
\xymatrix@C=5mm@R=5mm{
0\ar[r]&\mathbf{I}_k^{\infty}\mathbf{F}^\infty_kL_y\langle -j\rangle\ar[r]\ar@{=}[d]&
\mathbf{I}_k^{\infty}\mathbf{F}^\infty_kY_1\ar[r]\ar[d]&\dots\ar[r]& 
\mathbf{I}_k^{\infty}\mathbf{F}^\infty_kY_i\ar[r]\ar[d]& 
\mathbf{I}_k^{\infty}\mathbf{F}^\infty_kL_x\ar[r]\ar@{=}[d]& 0\\
0\ar[r]&\mathbf{I}_k^{\infty}\mathbf{F}^\infty_kL_y\langle -j\rangle\ar[r]&
\mathbf{I}_k^{\infty}Z_1\ar[r]&\dots\ar[r]& 
\mathbf{I}_k^{\infty}Z_i\ar[r]& 
\mathbf{I}_k^{\infty}\mathbf{F}^\infty_kL_x\ar[r]& 0\\
0\ar[r]&\mathbf{I}_k^{\infty}\mathbf{F}^\infty_kL_y\langle -j\rangle\ar[r]\ar@{=}[u]&
\mathbf{I}_k^{\infty}\mathbf{F}^\infty_kX_1\ar[r]\ar[u]&\dots\ar[r]& 
\mathbf{I}_k^{\infty}\mathbf{F}^\infty_kX_i\ar[r]\ar[u]& 
\mathbf{I}_k^{\infty}\mathbf{F}^\infty_kL_x\ar[r]\ar@{=}[u]& 0\\
}
\end{equation}
over ${}^{\mathbb{Z}}\mathcal{O}_0^{(\infty)}$ with exact rows. 

Set $\Lambda:=S_\infty\setminus S_k$.
Let $\mathbf{G}$ be the functor of taking the trace
of all projectives $P(w\cdot 0)^{(\infty)}$, where $w\in \Lambda$, in
a given module. Note that, in general, this functor is neither left nor
right exact, despite of the fact that it preserves both monomorphisms
and epimorphisms, see \cite{GM}. 

\begin{lemma}\label{lem-cp5}
Applying  $\mathbf{G}$ to \eqref{eq-cp6}, outputs a 
commutative diagram with exact rows.
\end{lemma}

\begin{proof}
We only need to prove that $\mathbf{G}$, applied to a row of \eqref{eq-cp6},
outputs an exact sequence.
Splitting longer exact sequences into shorter, it is enough to
prove the claim in the special case $i=1$, that is for short
exact sequences. If the original short exact sequence splits,
the claim is clear as each $\mathbf{I}_k^{\infty}L_w^{(k)}$
has simple top $L(w\cdot 0)^{\infty}$ and all other subquotients
are indexed by elements of $\Lambda$. Hence 
the application of $\mathbf{G}$ sends $\mathbf{I}_k^{\infty}L_w^{(k)}$
to its radical. In particular, $\mathbf{G}$ acts on the whole
short exact sequence as taking the radical.

The case when the original short exact sequence does not split
corresponds to non-vanishing of some $\mu(u,w)$. Since the value
of $\mu(u,w)$ does not depend on $k$, for all $k$ for which the value makes 
sense, the non-split extension of some 
$L_w^{(k)}$ and $L_u^{(k)}\langle -1\rangle$
gives a non-split extension of $L(w\cdot 0)^{\infty}$ and 
$L(u\cdot 0)^{\infty}\langle -1\rangle$ which is realized in the middle term of
our short exact sequence with end terms $\mathbf{I}_k^{\infty}L_w^{(k)}$ 
and $\mathbf{I}_k^{\infty}L_u^{(k)}\langle -1\rangle$ by the same arguments as
in the proof of Proposition~\ref{prop71}. This implies that the
application of $\mathbf{G}$ to the non-split extension of
$\mathbf{I}_k^{\infty}L_w^{(k)}$ and $\mathbf{I}_k^{\infty}L_u^{(k)}\langle -1\rangle$
outputs the extensions of the radicals of these two modules.
The claim of the lemma follows.
\end{proof}

Given Lemma~\ref{lem-cp5}, we can take the component-wise
quotient of \eqref{eq-cp6} by the subdiagram obtained by
applying $\mathbf{G}$ to \eqref{eq-cp6}. This will now be
a commutative diagram with exact rows over ${}^\mathbb{Z}\mathcal{C}^{\mathrm{res}}$.
Evaluating the adjunction morphism for the adjoint pair
$(\mathbf{I}_k^{\infty},\mathbf{F}_k^{\infty})$ and comparing characters, we 
obtain that the first row of the resulting diagram 
is isomorphic to \eqref{eq-cp1} while the third row  is isomorphic to \eqref{eq-cp2}.
By \cite[Theoreme~1, A X.121]{Bo}, we thus obtain that 
the images of  \eqref{eq-cp1} and \eqref{eq-cp2} in 
$\mathrm{ext}^{i}_{\mathcal{C}^{\mathrm{res}}}(L_x,L_y\langle -j\rangle)$
must coincide. Since \eqref{eq-cp1} and \eqref{eq-cp2} were arbitrary,
we obtain that this extension space is zero. This completes the proof.
\end{proof}

\vspace{2mm}

\noindent
Department of Mathematics, Uppsala University, Box. 480,
SE-75106, Uppsala, \\ SWEDEN, 
emails:
{\tt samuel.creedon\symbol{64}math.uu.se}\hspace{5mm}
{\tt mazor\symbol{64}math.uu.se}


\begin{thebibliography}{9999999}
% % \bibitem[AS03]{AS} 
% % Andersen, H.; Stroppel, C. Twisting functors on $\mathcal{O}$. 
% % Represent. Theory {\bf 7} (2003), 681--699. 
% % \bibitem[Ba01]{Ba} Backelin, E.
% % The Hom-spaces between projective functors.
% % Represent. Theory {\bf 5} (2001), 267--283.
% % % \bibitem[BV82]{BV1} Barbasch, D.; Vogan, D. Primitive ideals and 
% % % orbital integrals in complex classical groups. 
% % % Math. Ann. {\bf 259} (1982), no. 2, 153--199.
% % % \bibitem[BV83]{BV2} Barbasch, D.; Vogan, D. 
% % % Primitive ideals and orbital integrals in complex exceptional groups. 
% % % J. Algebra {\bf 80} (1983), no. 2, 350--382.
% % \bibitem[BB81]{BB} Beilinson, A.; Bernstein, J. 
% % Localisation de $\mathfrak{g}$-modules. 
% % C. R. Acad. Sci. Paris S{\'e}r. I Math. {\bf 292} (1981), no. 1, 15--18.
\bibitem[BG80]{BG} Bernstein, I.; Gelfand, S.;
Tensor products of finite- and infinite-dimensional 
representations of semisimple Lie algebras. Compositio Math. 
{\bf 41} (1980), no. 2, 245--285.
\bibitem[BGG76]{BGG} Bernstein, I.; Gelfand, I.; Gelfand, S. 
A certain category of $\mathfrak{g}$-modules. 
Funkcional. Anal. i Prilozhen. {\bf 10} (1976), no. 2, 1--8.
% % \bibitem[BB05]{BjBr}  Bj{\"o}rner, A.; Brenti, F.
% % Combinatorics of Coxeter groups
% % Grad. Texts in Math., {\bf 231}, Springer, New York, 2005, xiv+363 pp.
% % \bibitem[BW01]{BW} Billey, S.; Warrington, G.
% % Kazhdan-Lusztig polynomials for $321$-hexagon-avoiding permutations
% % J. Algebraic Combin. {\bf 13} (2001), no. 2, 111--136.
\bibitem[BB]{BB}
Bj{\"o}rner, A.; Brenti, F.
Combinatorics of Coxeter groups.
Grad. Texts in Math., {\bf 231}
Springer, New York, 2005, xiv+363 pp.
% % \bibitem[B-W22]{BBDVW} Blundell, C.; Buesing, L.; Davies, A.; 
% % Veli{\v c}kovi{\'c}, P.; Williamson, G.
% % Towards combinatorial invariance for Kazhdan-Lusztig polynomials.
% % Represent. Theory {\bf 26} (2022), 1145--1191.
\bibitem[Bo07]{Bo} Bourbaki. N.
Elements of mathematics. Commutative algebra. Chapter 10.
Springer-Verlag, Berlin, 2007, ii+187 pp.
% % % \bibitem[BS10]{BS1} Brundan, J.; Stroppel, C. Highest weight categories arising 
% % % from Khovanov’s diagram algebra. II. Koszulity. Transform. Groups {\bf 15} 
% % % (2010), no. 1, 1--45.
% % % \bibitem[BS11]{BS2} Brundan, J.; Stroppel, C. Highest weight categories arising 
% % % from Khovanov’s diagram algebra III: category O. Represent. Theory {\bf 15} (2011), 170--243.
% % \bibitem[BK81]{BK} Brylinski, J.-L.; Kashiwara, M. Kazhdan-Lusztig 
% % conjecture and holonomic systems. Invent. Math. {\bf 64} (1981), no. 3, 387--410. 
% \bibitem[Ch24]{Ch24}
% Chen, C.-W. Kostant's problem for Whittaker modules
% Lett. Math. Phys. {\bf 114} (2024), no. 1, Paper No. 10, 34 pp.
% % \bibitem[CPS88]{CPS}
% % Cline, E.; Parshall, B.; Scott, L.
% % Finite-dimensional algebras and highest weight categories.
% % J. Reine Angew. Math. {\bf 391} (1988), 85--99.
% % \bibitem[CM17]{CM}
% % Coulembier, K.; Mazorchuk, V.
% % Dualities and derived equivalences for category $\mathcal{O}$.
% % Israel J. Math. {\bf 219} (2017), no. 2, 661--706.
\bibitem[CMZ19]{CMZ}
Coulembier, K.; Mazorchuk, V.; Zhang, X.
Indecomposable manipulations with simple modules in category $\mathcal{O}$.
Math. Res. Lett. {\bf 26} (2019), no. 2, 447--499.
\bibitem[CP19]{CP}
Coulembier, K.; Penkov, I.
On an infinite limit of BGG categories $\mathcal{O}$.
Mosc. Math. J. {\bf 19} (2019), no. 4, 655--693.
\bibitem[CM24]{CM1}
Creedon, S., Mazorchuk, V.
Almost all permutations and involutions are Kostant negative.
Preprint arXiv:2411.13043, to appear in Math. Pannonica.
\bibitem[CM25]{CM2}
Creedon, S., Mazorchuk, V.
Consecutive Patterns, Kostant's Problem and Type $A_6$.
Preprint arXiv:2503.07809.
% % % \bibitem[Di96]{Di} Dixmier, J. Enveloping algebras. Graduate Studies in Mathematics, 
% % % {\bf 11}. American Mathematical Society, Providence, RI, 1996. xx+379 pp.
% % % \bibitem[Du75]{Du} Duflo, M. Construction of primitive ideals in an enveloping 
% % % algebra. Lie groups and their representations (Proc. Summer School, 
% % % Bolyai Janos Math. Soc., Budapest, 1971), pp. 77--93. Halsted, New York, 1975.
% % % \bibitem[EW14]{EW} Elias, B.; Williamson, G. The Hodge theory of Soergel 
% % % bimodules. Ann. of Math. (2) {\bf 180} (2014), no. 3, 1089--1136.
% \bibitem[CM24a]{CM1} Creedon, S.; Mazorchuk, V.; Consecutive patterns,
% Kostant problem and type $A_6$. In preparation.
% \bibitem[CM24b]{CM2} Creedon, S.; Mazorchuk, V.; Kostant problem
% for longest elements in parabolic subgroups. In preparation.
\bibitem[EA15]{EA15} Entova-Aizenbud,~I.
Notes on Restricted Inverse Limits of Categories.
Preprint arXiv:1504.01121.
% % \bibitem[GJ81]{GJ} Gabber, O.; Joseph, A. On the Bernstein-Gelfand-Gelfand 
% % resolution and the Duflo sum formula. Compositio Math. {\bf 43} (1981), no. 1, 107--131.
\bibitem[Ge06]{Ge} Geck, M. Kazhdan-Lusztig cells and the Murphy basis.
Proc. London Math. Soc. (3) {\bf 93} (2006), no. 3, 635--665.
\bibitem[GM17]{GM} Grensing, A.-L.; Mazorchuk, V.
Finitary 2-categories associated with dual projection functors.
Commun. Contemp. Math. {\bf 19} (2017), no. 3, 1650016, 40 pp.
% \bibitem[GPP01]{GPP} Guibert, O.; Pergola, E.; Pinzani, R.
% Vexillary involutions are enumerated by Motzkin numbers.
% Ann. Comb. {\bf 5} (2001), no. 2, 153--174.
% % % \bibitem[HC51]{HC} Harish-Chandra. On some applications of the universal 
% % % enveloping algebra of a semisimple Lie algebra. Trans. Amer. Math. Soc. {\bf 70} 
% % % (1951), 28--96.
\bibitem[HHS21]{HHS}
He, Z.; Hu, J.; Sun, Y.
Kazhdan-Lusztig left cell preorder and dominance order.
Preprint arXiv:2109.13646.
\bibitem[Hu08]{Hu} Humphreys, J.
Representations of semisimple Lie algebras in the BGG category $\mathcal{O}$.
Grad. Stud. Math., {\bf 94}
American Mathematical Society, Providence, RI, 2008, xvi+289 pp.
% % \bibitem[Ir85]{Ir} Irving, R.
% % Projective modules in the category $\mathcal{O}_S$: self-duality
% % Trans. Amer. Math. Soc. {\bf 291} (1985), no. 2, 701--732.
% % \bibitem[Ir90]{Ir90} Irving, R.
% % A filtered category $\mathcal{O}_S$ and applications
% % Mem. Amer. Math. Soc. {\bf 83} (1990), no. 419, vi+117 pp.
% % \bibitem[Ja83]{Ja} Jantzen, J. Einh{\"u}llende Algebren halbeinfacher Lie-Algebren.
% % Ergebnisse der Mathematik und ihrer Grenzgebiete (3), {\bf 3}. 
% % Springer-Verlag, Berlin, 1983. ii+298 pp.
\bibitem[Jo80]{Jo} Joseph, A. Kostant's problem, Goldie rank and the Gelfand-Kirillov 
conjecture. Invent. Math. {\bf 56} (1980), no. 3, 191--213.
\bibitem[KL79]{KL} Kazhdan, D.; Lusztig, G. Representations of Coxeter groups 
and Hecke algebras. Invent. Math. {\bf 53} (1979), no. 2, 165--184.
% % \bibitem[KiM16]{KiM} Kildetoft, T.; Mazorchuk, V.
% % Parabolic projective functors in type A. Adv. Math. {\bf 301} (2016), 785--803.
% % % \bibitem[KM04]{KM0} Khomenko, O.; Mazorchuk, V.
% % % Structure of modules induced from simple modules with minimal annihilator.
% % % Canad. J. Math. {\bf 56} (2004), no.2, 293--309.
% % \bibitem[KM05]{KM} Khomenko, O.; Mazorchuk, V. On Arkhipov's and Enright's 
% % functors. Math. Z. {\bf 249} (2005), no. 2, 357--386. 
% \bibitem[Kn98]{Kn} Knuth, D. The Art of Computer Programming, 
% volume 3: Sorting and Searching; Second Edition, Addison-Wesley, 1998.
% % % \bibitem[KMM21]{KMM2} Ko, H.; Mazorchuk, V.; Mr{\dj}en, R.
% % % Bigrassmannian permutations and Verma modules.
% % % Selecta Math. (N.S.) {\bf 27} (2021), no.4, Paper No. 55, 24 pp.
\bibitem[KMM23]{KMM} Ko, H.; Mazorchuk, V.; Mr{\dj}en, R. Some homological 
properties of category $\mathcal{O}$, V. Int. Math. Res. Not. IMRN 
{\bf 2023}, no. 4, 3329--373.
% % % \bibitem[Ko63]{Ko} Kostant, B. Lie group representations on polynomial rings. 
% % % Amer. J. Math. {\bf 85} (1963), 327--404.
% % \bibitem[KSX01]{KSX} K{\"o}nig, S.; Slung{\aa}rd, I.; Xi, C. 
% % Double centralizer properties, dominant dimension, and tilting modules. 
% % J. Algebra {\bf 240} (2001), no. 1, 393--412.
% % \bibitem[KX99]{KoXi} K{\"o}nig, S.; Xi, C.
% % When is a cellular algebra quasi-hereditary?
% % Math. Ann. {\bf 315} (1999), no. 2, 281--293.
% % \bibitem[Ka10]{Ka} K{\aa}hrstr{\"o}m, J. Kostant's problem and parabolic 
% % subgroups. Glasg. Math. J. {\bf 52} (2010), no. 1, 19--32.
% % \bibitem[KaM10]{KaM} K{\aa}hrstr{\"o}m, J.; Mazorchuk, V. A new approach to 
% % Kostant's problem. Algebra Number Theory {\bf 4} (2010), no. 3, 231--254.
% % \bibitem[Lu87]{Lu} Lusztig, G. Cells in affine Weyl groups. II. 
% % J. Algebra {\bf 109} (1987), 536--548.
% \bibitem[MMM24]{MMM} Mackaay, M.; Mazorchuk, V.; Miemietz, V.
% Kostant's problem for fully commutative permutations,
% Rev. Mat. Iberoam. {\bf 40} (2024), no. 2, 537--563.
% % \bibitem[M-Z23]{MMMTZ}
% % Mackaay, M.; Mazorchuk, V.; Miemietz, V.; 
% % Tubbenhauer, D.; Zhang, X. Simple transitive 
% % 2-representations of Soergel bimodules for 
% % finite Coxeter types. Proc. Lond. Math. Soc. (3) 
% % {\bf 126} (2023), no. 5, 1585--1655.
% \bibitem[MT04]{MT} Marcus, A.; Tardos, G.
% Excluded permutation matrices and the Stanley-Wilf conjecture.
% J. Combin. Theory Ser. A {\bf 107} (2004), no. 1, 153--160.
% % % \bibitem[Mat64]{Mat} Matsumoto, H. G{\'e}n{\'e}rateurs et relations 
% % % des groupes de Weyl g{\'e}n{\'e}ralis{\'e}s. C. R. Acad. Sci.
% % % Paris {\bf 258} (1964), 3419--3422.
% % \bibitem[Ma05]{Ma} Mazorchuk, V. A twisted approach to Kostant's problem. 
% % Glasg. Math. J. {\bf 47} (2005), no. 3, 549--561.
% % \bibitem[Ma10]{Ma2} Mazorchuk, V. Some homological properties of the category 
% % O. II. Represent. Theory {\bf 14} (2010), 249--263.
% \bibitem[Ma23]{Ma23} Mazorchuk, V. The tale of Kostant's problem.
% Preprint arXiv:2308.02839. To appear in the Proceedings of the XIV Ukrainian
% Algebraic Conference.
\bibitem[MM11]{MM1} Mazorchuk, V.; Miemietz, V. Cell 2-representations of 
finitary 2-categories. Compos. Math. {\bf 147} (2011), no. 5, 1519--1545.
% % \bibitem[MSr23]{MSr} Mazorchuk, V.; Srivastava S.
% % Kostant's problem for parabolic Verma modules. Preprint arXiv:2301.07090.
% % To appear in Proc. Edinburgh Math. Soc.
% \bibitem[MS08a]{MS} Mazorchuk, V.; Stroppel, C. Categorification of Wedderburn's 
% basis for $\mathbb{C}[S_n]$. Arch. Math. (Basel) {\bf 91} (2008), no. 1, 1--11.
\bibitem[MS08]{MS} Mazorchuk, V.; Stroppel, C. Categorification of (induced) 
cell modules and the rough structure of generalised Verma modules.
Adv. Math. {\bf 219} (2008), no.4, 1363--1426.
% % \bibitem[MSt08b]{MS} Mazorchuk, V.; Stroppel, C.
% % Projective-injective modules, Serre functors and symmetric algebras.
% % J. Reine Angew. Math. {\bf 616} (2008), 131--165.
% % % \bibitem[MT22]{MT} Mazorchuk, V.; Tenner, B.
% % % Intersecting principal Bruhat ideals and grades of simple modules.
% % % Comb. Theory {\bf 2} (2022), no.1, Paper No. 14, 31 pp.
% \bibitem[MiSo97]{MiSo}  Mili{\v c}i{\'c}, D.; Soergel, W.
% The composition series of modules induced from Whittaker modules.
% Comment. Math. Helv. {\bf 72} (1997), no.4, 503--520.
\bibitem[Na17a]{Na1} Nampaisarn, T.
Categories $\mathcal{O}$ for Dynkin Borel Subalgebras of 
Root-Reductive Lie Algebras.
Preprint arXiv:1706.05950.
\bibitem[Na17b]{Na2} Nampaisarn, T.
On Categories $\mathcal{O}$ for Root-Reductive Lie Algebras.
Preprint arXiv:1711.11234.

\bibitem[Na20]{Na3} Nampaisarn, T.
Categories $\mathcal{O}$ for Root-Reductive Lie Algebras: 
II. Translation Functors and Tilting Modules.
Preprint arXiv:2012.01003.

% % \bibitem[RC80]{RC} Rocha-Caridi, A. Splitting criteria for $\mathfrak{g}$-modules 
% % induced from a parabolic and the Bernstein-Gelfand-Gelfand resolution 
% % of a finite-dimensional, irreducible $\mathfrak{g}$-module. 
% % Trans. Amer. Math. Soc. {\bf 262} (1980), no. 2, 335--366.
\bibitem[Sa01]{Sa} Sagan, B. The symmetric group. Representations, combinatorial 
algorithms, and symmetric functions. Second edition. Graduate Texts in Mathematics, 
{\bf 203}. Springer-Verlag, New York, 2001. xvi+238 pp.
\bibitem[Sc61]{Sch} Schensted, C. Longest increasing and decreasing subsequences.
Canadian J. Math. {\bf 13} (1961), 179--191.
% % \bibitem[SSV98]{SSV} Shapiro, B.; Shapiro, M.; Vainshtein, A.
% % Kazhdan-Lusztig polynomials for certain varieties of incomplete flags.
% % Discrete Math. {\bf 180} (1998), no. 1-3, 345--355.
% \bibitem[OEIS]{OEIS} N. J. A. Sloane.
% The Online Encyclopedia of Integer Sequences.
% Founded in 1964.
\bibitem[So90]{So0} Soergel, W. Kategorie $\mathcal{O}$, perverse Garben und 
Moduln {\"u}ber den Koinvarianten zur Weylgruppe.
J. Amer. Math. Soc. {\bf 3} (1990), no.2, 421--445.
\bibitem[So92]{So92}	Soergel, W. 
The combinatorics of Harish-Chandra bimodules.
J. Reine Angew. Math. {\bf 429} (1992), 49--74.category o gradings
\bibitem[So97]{So}	Soergel, W. Kazhdan-Lusztig polynomials and 
a combinatoric for tilting modules. Represent. Theory {\bf 1} (1997), 83--114. 
% % \bibitem[So97]{So3} Soergel, W.  Charakterformeln f{\"u}r Kipp-Moduln 
% % {\"u}ber Kac-Moody-Algebren. Represent. Theory {\bf 1} (1997), 115--132.
% % \bibitem[So07]{So2} Soergel, W. Kazhdan-Lusztig-Polynome und unzerlegbare 
% % Bimoduln {\"u}ber Polynomringen. 
% % J. Inst. Math. Jussieu {\bf 6} (2007), no. 3, 501--525.
\bibitem[St03]{St0} Stroppel, C. 
Category $\mathcal{O}$: gradings and translation functors
J. Algebra {\bf 268} (2003), no. 1, 301--326.
% % \bibitem[St03]{St} Stroppel, C.
% % Category $\mathcal{O}$: quivers and endomorphism rings of projectives.
% % Represent. Theory {\bf 7} (2003), 322--345.
% % % \bibitem[Ve66]{Ve} Verma, D.-N. Structure of certain induced representations
% % % of complex semisimple Lie algebras. Thesis (Ph.D.), Yale University. 
% % % 1966. 107 pp.
% % \bibitem[Vi10]{Vi} R.~Virk. A remark on some bases in the Hecke algebra.
% % Preprint arXiv:1012.1924. 
\bibitem[Wa11]{Wa}
Warrington, G.
Equivalence classes for the $\mu$-coefficient of Kazhdan-Lusztig 
polynomials in $S_n$.
Exp. Math. {\bf 20} (2011), no. 4, 457--466.
\end{thebibliography}
\end{document}